\documentclass[12pt]{article}

\usepackage{amsfonts,amsmath,amsthm,amssymb,hyperref,mathtools,float,tabularx}

\newcommand{\ppt}{\operatorname{ppt}}
\newcommand{\zpth}{\operatorname{th}_{\gamma_P}}
\newcommand{\zpt}{\operatorname{pt}}
\newcommand{\zth}{\operatorname{th}}

\newcommand{\pd}{\mathcal{G}}

\renewenvironment{thebibliography}[1]{
  \begin{oldthebibliography}{#1}
    \setlength{\itemsep}{0em}
    \setlength{\parskip}{0em}
}
{
  \end{oldthebibliography}
}

\usepackage{geometry}
\newtheorem{theorem}{Theorem}
\newtheorem{corollary}[theorem]{Corollary}
\newtheorem{lemma}[theorem]{Lemma}
\newtheorem{proposition}[theorem]{Proposition}

\theoremstyle{definition}
\newtheorem{definition}{Definition}

\usepackage{color}

\definecolor{red}{rgb}{1,0,0}

\definecolor{blue}{rgb}{0,0,1}

\begin{document}

\title{Power domination throttling}

\author{Boris Brimkov\thanks{Department of Computational and Applied Mathematics, Rice University, Houston, TX, 77005, USA (boris.brimkov@rice.edu, ivhicks@rice.edu, rsp7@rice.edu, logan.smith@rice.edu)} \hskip 1.5em 
Joshua Carlson\thanks{Department of Mathematics, Iowa State University, Ames, IA, 50011, USA (jmsdg7@iastate.edu)} \hskip 1.5em
Illya V. Hicks$^*$\\ 
Rutvik Patel$^*$ \hskip 1.5em
Logan Smith$^*$}
\date{}

\maketitle

%
%

\begin{abstract}

A power dominating set of a graph $G=(V,E)$ is a set $S\subset V$ that colors every vertex of $G$ according to the following rules: in the first timestep, every vertex in $N[S]$ becomes colored; in each subsequent timestep, every vertex which is the only non-colored neighbor of some colored vertex becomes colored. The power domination throttling number of $G$ is the minimum sum of the size of a power dominating set $S$ and the number of timesteps it takes $S$ to color the graph. In this paper, we determine the complexity of power domination throttling and give some tools for computing and bounding the power domination throttling number. Some of our results apply to very general variants of throttling and to other aspects of power domination.

\smallskip

\end{abstract}

\section{Introduction}
A \emph{power dominating set} of a graph $G=(V,E)$ is a set $S\subset V$ that colors every vertex of $G$ according to the following rules: in the first timestep, every vertex in $N[S]$ becomes colored; in each subsequent timestep, every vertex which is the only non-colored neighbor of some colored vertex becomes colored. The first timestep is called the \emph{domination step} and each subsequent timestep is called a \emph{forcing step}. The {\em power domination number} of $G$, denoted $\gamma_P(G)$, is the cardinality of a minimum power dominating set. The \emph{power propagation time of $G$ using $S$}, denoted $\ppt(G;S)$, is the number of timesteps it takes for a power dominating set $S$ to color all of $G$. 
The \emph{power propagation time} of $G$ is defined as 
\[\ppt(G)=\min\{\ppt(G;S):S \text{ is a minimum power dominating set}\}.\] 
It is well-known that larger power dominating sets do not necessarily yield smaller power propagation times.
The \emph{power domination throttling number} of $G$ is defined as  
\[\zpth(G)=\min\{|S|+\ppt(G;S):S \text{ is a power dominating set}\}.\]
$S$ is a \emph{power throttling set} of $G$ if $S$ is a power dominating set of $G$ and $|S|+\ppt(G;S)=\zpth(G)$.

Power domination arises from a graph theoretic model of the Phase Measurement Unit (PMU) placement problem from electrical engineering. Electrical power companies place PMUs at select locations in a power network in order to monitor its performance; the physical laws by which PMUs observe the network give rise to the color change rules described above (cf. \cite{BH05,powerdom3}). This PMU placement problem has been explored extensively in the electrical engineering literature; see \cite{EEprobabilistic,Baldwin93,Brunei93,EEinformationTheoretic,EEtaxonomy,Mili91,EEtabuSearch,EEmultiStage}, and the bibliographies therein for various placement strategies and computational results. The PMU placement literature also considers various other properties of power dominating sets, such as redundancy, controlled islanding, and connectedness, and optimizes over them in addition to the cardinality of the set (see, e.g., \cite{akhlaghi,connected_pd,mahari,xia}).

Power domination has also been widely studied from a purely graph theoretic perspective. See, e.g., \cite{benson,bozeman,connected_pd,dorbec2,dorfling,kneis,xu,powerdom2} for various structural and computational results about power domination and related variants. The power propagation time of a graph has previously been studied in \cite{aazami,dorbec,ferrero,liao}.  
Other variants of propagation time arising from similar dynamic graph coloring processes have also been studied; these include zero forcing propagation time \cite{berliner,fast_hicks,proptime1,kenter} and positive semidefinite propagation time \cite{warnberg}. Throttling for other problems such as zero forcing \cite{butler_young}, positive semidefinite zero forcing \cite{carlson2}, minor monotone floor of zero forcing \cite{carlson}, and the game of Cops and Robbers \cite{breen} has been studied as well.

Notably missing from the literature on throttling (for power domination as well as other variants) is the computational complexity of the problems. In this paper, we determine the complexity of a large, abstract class of throttling problems, including power domination throttling. We also give explicit formulas and tight bounds for the power domination throttling numbers of certain graphs, and characterizations of graphs with extremal power domination throttling numbers.

\section{Preliminaries}

A graph $G=(V,E)$ consists of a vertex set $V=V(G)$ and an edge set $E=E(G)$ of two-element subsets of $V$. The \emph{order} of $G$ is denoted by $n(G)=|V|$. We will assume that the order of $G$ is nonzero, and when there is no scope for confusion, dependence on $G$ will be omitted. Two vertices $v,w\in V$ are \emph{adjacent}, or \emph{neighbors}, if $\{v,w\}\in E$; we will sometimes write $vw$ to denote an edge $\{v,w\}$. The \emph{neighborhood} of $v\in V$ is the set of all vertices which are adjacent to $v$, denoted $N(v)$; the \emph{degree} of $v\in V$ is defined as $d(v)=|N(v)|$. The \emph{maximum degree} of $G$ is defined as $\Delta(G)=\max_{v\in V}d(v)$; when there is no scope for confusion, dependence on $G$ will be omitted. The \emph{closed neighborhood} of $v$ is the set $N[v]=N(v)\cup \{v\}$. 

\emph{Contracting} an edge $e$ of a graph $G$, denoted $G/e$, is the operation of removing $e$ from $G$ and identifying the endpoints of $e$ into a single vertex. A graph $H$ is a \emph{subgraph} of a graph $G$, denoted $H\leq G$, if $H$ can be obtained from $G$ by deleting vertices and deleting edges of $G$; $H$ is a \emph{minor} of $G$, denoted $H\preceq G$, if $H$ can be obtained from $G$ by deleting vertices, deleting edges, and contracting edges of $G$.
Given $S \subset V$, $N[S]=\bigcup_{v\in S}N[v]$, and the \emph{induced subgraph} $G[S]$ is the subgraph of $G$ whose vertex set is $S$ and whose edge set consists of all edges of $G$ which have both endpoints in $S$. An isomorphism between graphs $G_1$ and $G_2$ will be denoted by $G_1\simeq G_2$. 
Given a graph $G=(V,E)$, and sets $V'\subset V$ and $E'\subset E$, we say the vertices in $V'$ are \emph{saturated} by the edges in $E'$ if every vertex of $V'$ is incident to some edge in $E'$. 
An \emph{isolated vertex}, or \emph{isolate}, is a vertex of degree 0. A \emph{dominating vertex} is a vertex which is adjacent to all other vertices. The path, cycle, complete graph, and empty graph on $n$ vertices will respectively be denoted $P_n$, $C_n$, $K_n$, $\overline{K}_n$.

Given two graphs $G_1$ and $G_2$, the \emph{disjoint union} $G_1\dot\cup G_2$ is the graph with vertex set $V(G_1)\dot\cup V(G_2)$ and edge set $E(G_1)\dot\cup E(G_2)$. With a slight abuse in notation, given a set $S\subset V(G_1\dot\cup G_2)$, we will use, e.g., $S\cap V(G_1)$ to denote the set of vertices in $G_1\dot\cup G_2$ originating from $G_1$ (instead of specifying the unique index created by the disjoint union operation). The \emph{intersection} of $G_1$ and $G_2$, denoted $G_1\cap G_2$, is the graph with vertex set $V(G_1)\cap V(G_2)$ and edge set $E(G_1)\cap E(G_2)$.
The \emph{Cartesian product} of $G_1$ and $G_2$, denoted $G_1\square G_2$, is the graph with vertex set $V(G_1)\times V(G_2)$, where vertices $(u,u')$ and $(v,v')$ are adjacent in $G_1\square G_2$ if and only if either $u = v$ and $u'$ is adjacent to $v'$ in $G_2$, or $u' = v'$ and $u$ is adjacent to $v$ in $G_1$. 
The \emph{join} of $G_1$ and $G_2$, denoted $G_1\lor G_2$, is the graph obtained from $G_1\dot\cup G_2$ by adding an edge between each vertex of $G_1$ and each vertex of $G_2$. 
The \emph{complete bipartite graph} with parts of size $a$ and $b$, denoted $K_{a,b}$, is the graph $\overline{K}_a\lor\overline{K}_b$. The graph $K_{n-1,1}$, $n\geq 3$, will be called a \emph{star}. 
For other graph theoretic terminology and definitions, we refer the reader to \cite{bondy}.

A \emph{zero forcing} set of a graph $G=(V,E)$ is a set $S\subset V$ that colors every vertex of $G$ according to the following color change rule: initially, every vertex in $S$ is colored; then, in each timestep, every vertex which is the only non-colored neighbor of some colored vertex becomes colored. Note that in a given forcing step, it may happen that a vertex $v$ is the only non-colored neighbor of several colored vertices. In this case, we may arbitrarily choose one of those colored vertices $u$, and say that $u$ is the one which forces $v$; making such choices in every forcing step will be called ``fixing a chronological list of forces". 
The notions of \emph{zero forcing number} of $G$, denoted $Z(G)$, \emph{zero forcing propagation time of $G$ using $S$}, denoted $\zpt(G;S)$, \emph{zero forcing propagation time} of $G$, denoted $\zpt(G)$, and \emph{zero forcing throttling number}, denoted $\zth(G)$, are defined analogously to $\gamma_P(G)$, $\ppt(G;S)$, $\ppt(G)$, and $\zpth(G)$. A \emph{positive semidefinite (PSD) zero forcing set} of $G$ is a set $S\subset V$ which colors every vertex of $G$ according to the following color change rule: initially, in timestep 0, every vertex in $S_0:=S$ is colored; then, in each timestep $i\geq 1$, if $S_{i-1}$ is the set of colored vertices in timestep $i-1$, and $W_1,\ldots,W_k$ are the vertex sets of the components of $G-S_{i-1}$, then every vertex which is the only non-colored neighbor of some colored vertex in $G[W_j\cup S_{i-1}]$, $1\leq j\leq k$, becomes colored. As with zero forcing, the PSD zero forcing notation $Z_+(G)$, $\zpt_+(G;S)$, $\zpt_+(G)$, and $\zth_+(G)$ is analogous to $\gamma_P(G)$, $\ppt(G;S)$, $\ppt(G)$, and $\zpth(G)$, respectively. For every graph $G$, $\gamma_P(G)\le\zpth(G)\le\zth(G)$. Moreover, in general, $\zpth(G)$ and $\zth_+(G)$ are not comparable; for example, $\zpth(K_7)<th_+(K_7)$, while $\zpth(G)>th_+(G)$ for $G=(\{1,2,3,4,5,6,7\},\{\{1,2\},\{1,3\},\{1,4\},\{2,5\},\{2,6\},\{3,7\}\})$.

\section{Complexity Results}

A number of NP-Completeness results have been presented for power domination, zero forcing, and positive semidefinite zero forcing. For example power domination was shown to be NP-Complete for general graphs \cite{powerdom3}, planar graphs \cite{guo}, chordal graphs \cite{powerdom3}, bipartite graphs \cite{powerdom3}, split graphs \cite{guo,liao2}, and circle graphs \cite{guo}; zero forcing was shown to be NP-Complete for general graphs \cite{aazami2,fallat} and planar graphs \cite{aazami2}; PSD zero forcing was shown to be NP-complete for general graphs \cite{yang} and line graphs \cite{wang}. However, despite recent interest in the corresponding throttling problems, to our knowledge there are no complexity results for any of those problems. In this section, we provide sufficient conditions which ensure that given an NP-Complete vertex minimization problem, the corresponding throttling problem is also NP-Complete.

To facilitate the upcoming discussion, we recall three categories of graph parameters introduced by Lov\'{a}sz \cite{lovasz}. Let $\phi$ be a graph parameter and $G_1$ and $G_2$ be two graphs on which $\phi$ is defined. Then, $\phi$ is called {\em maxing} if $\phi(G_1 \dot\cup G_2) = \max\{\phi(G_1), \phi(G_2)\}$, {\em additive} if $\phi(G_1 \dot\cup G_2) = \phi(G_1) + \phi(G_2)$, and {\em multiplicative} if $\phi(G_1 \dot\cup G_2) = \phi(G_1)\phi(G_2)$. For example, $\gamma_P(G)$ is an additive parameter, $\ppt(G)$ is a maxing parameter, and the number of distinct power dominating sets admitted by $G$ is a multiplicative parameter. We will show that with only minor additional assumptions, a minimization problem defined as the sum of a maxing parameter and an additive parameter inherits the NP-Completeness of the additive parameter for any family of graphs. 

\begin{definition}
\label{def1}
Given a graph $G=(V,E)$, let $X(G)$ be a set of subsets of $V$ and let $p(G;\,\cdot\,)$ be a function which maps a member of $X(G)$ to a nonnegative integer. Define the parameters $x(G):=\min_{S \in X(G)} |S|$ and $p(G):=\min_{\substack{S \in X(G)\\ |S| = x(G)}} p(G;S)$, and define $\arg p(G):=\arg\min_{\substack{S \in X(G)\\ |S| = x(G)}} p(G;S)$.
\end{definition}

\noindent Note that the function $p$ and the parameter $p$ are differentiated by their inputs. Table \ref{table_exp} shows the power domination notation corresponding to the abstract notation of Definition~\ref{def1}.

\begin{table}[H]
\renewcommand{\arraystretch}{1.5}
\centering
\begin{tabularx}{\textwidth}{|X|X|}
\hline
	
\textbf{Abstract notation} &\textbf{Power domination notation}\\
\hline
$X(G)$ & Set of power dominating sets of $G$\\
$x(G)$ & $\gamma_P(G)$\\
$p(G;S)$ &$\ppt(G;S)$\\
$p(G)$ &$\ppt(G)$\\
$\min_{S\in X(G)} \{|S|+p(G;S)\}$ &$\zpth(G)$\\
\hline
\end{tabularx}
\caption{Notation for abstract problems and corresponding notation for power domination.}
\label{table_exp}

\end{table}

\noindent Table \ref{fig:probs} gives a pair of abstract decision problems that can be defined for $X$, $x$, and $p$, as well as three instances which have been studied in the literature.

\begin{table}[H]
\newcolumntype{S}{>{\hsize=.9\hsize}X}
\newcolumntype{D}{>{\hsize=1.1\hsize}X}
	\centering
	\renewcommand{\arraystretch}{1.5}
	\begin{tabularx}{\textwidth}{|S|D|}
	\hline
	\textbf{Set minimization problem} & \textbf{Throttling problem}
	\\ \hline

	\textsc{Minimum $X$ set} \newline
	\textbf{Instance:} Graph $G$, integer $k$ \newline
	\textbf{Question:} Is $x(G)<k$? &
    
	\textsc{$(X,p)$-Throttling} \newline
	\textbf{Instance:} Graph $G$, integer $k$ \newline
	\textbf{Question:} Is $\min_{S\in X(G)}\{|S|+p(G;S)\}<k$? \\ 
	\hline
	\hline
	\textsc{Power Domination} \newline
	\textbf{Instance:} Graph $G$, integer $k$ \newline
	\textbf{Question:} Is $\gamma_P(G)<k$? &
    
	\textsc{Power Domination Throttling} \newline
	\textbf{Instance:} Graph $G$, integer $k$ \newline
	\textbf{Question:} Is $\zpth(G)<k$? \\ \hline
    
	\textsc{Zero Forcing} \newline
	\textbf{Instance:} Graph $G$, integer $k$ \newline
	\textbf{Question:} Is $Z(G)<k$? &

	\textsc{Zero Forcing Throttling} \newline
	\textbf{Instance:} Graph $G$, integer $k$ \newline
	\textbf{Question:} Is $\zth(G)<k$? \\ \hline
    
	\textsc{PSD Zero Forcing} \newline
	\textbf{Instance:} Graph $G$, integer $k$ \newline
	\textbf{Question:} Is $Z_+(G)<k$? &

	\textsc{PSD Zero Forcing Throttling} \newline
	\textbf{Instance:} Graph $G$, integer $k$ \newline
	\textbf{Question:} Is $\zth_{+}(G)<k$? \\ \hline
       
	\end{tabularx}
    \caption{NP-Complete set minimization problems and corresponding throttling problems.}
    \label{fig:probs}
    
\end{table}

\noindent We now give sufficient conditions to relate the complexity of these problems.


\begin{theorem} \label{thm:comp}
Let $X$ and $p$ (as in Definition \ref{def1}) satisfy the following:
\begin{enumerate}
\item[1)] For any graph $G$, there exist constants $b$, $c$ such that for any set $S \in X(G)$, $p(G;S) < b = O(|V(G)|^{c})$, and $p(G;S)$ and $b$ can be computed in $O(|V(G)|^{c})$ time.
\item[2)] For any graphs $G_1$ and $G_2$, $X(G_1 \dot\cup G_2) = \{S_1 \dot\cup S_2 : S_1 \in X(G_1), S_2 \in X(G_2)\}$.
\item[3)] For any graphs $G_1$ and $G_2$, and for any $S_1 \in X(G_1)$ and $S_2 \in X(G_2)$, $p(G_1 \dot\cup G_2; S_1 \dot\cup S_2) = \max\{ p(G_1;S_1), p(G_2;S_2) \}$.
\item[4)] \textsc{Minimum $X$ Set} is NP-Complete.
\end{enumerate}
Then, \textsc{$(X,p)$-Throttling} is NP-Complete.
\end{theorem}

\begin{proof}
We will first show that $x$ is an additive parameter and $p$ is a maxing parameter. Let $G_1$ and $G_2$ be graphs. By 2), 
\begin{eqnarray*}
x(G_1 \dot\cup G_2)&=& \min \{|S'|:S' \in X(G_1 \dot\cup G_2)\}\\
&=&\min \{|S'|:S' \in \{S_1 \dot\cup S_2 : S_1 \in X(G_1), S_2 \in X(G_2)\}\}\\
&=&\min \{|S_1|+|S_2|:S_1 \in X(G_1), S_2\in X(G_2)\}\\
&=&\min\{|S_1|:S_1\in X(G_1)\}+\min\{|S_2|:S_2\in X(G_2)\}= x(G_1)+x(G_2).
\end{eqnarray*}
Thus, $x$ is additive by definition. Now let $S^*$ be a set in $\arg p(G_1 \dot\cup G_2)$. By 2), there exist sets $S_1\in X(G_1)$ and $S_2\in X(G_2)$ such that $S^*=S_1\dot\cup S_2$. By definition, $|S_1|\geq x(G_1)$ and $|S_2|\geq x(G_2)$, and since $x$ is additive, $|S^*|=x(G_1\dot\cup G_2)=x(G_1)+x(G_2)$. Thus, $|S_1|=x(G_1)$ and $|S_2|=x(G_2)$. Then,
\begin{eqnarray*}
p(G_1\dot\cup G_2)&=&p(G_1 \dot\cup G_2;S^*) = p(G_1 \dot\cup G_2; S_1 \dot\cup S_2) = \max \{p(G_1;S_1), p(G_2;S_2)\} \\
&\geq& \max \left\{\min_{\substack{S\in X(G_1) \\ |S| = x(G_1)}}p(G_1;S), \min_{\substack{S\in X(G_2) \\ |S| = x(G_2)}}p(G_2;S)\right\} =\max \{p(G_1), p(G_2) \},
\end{eqnarray*}
where the third equality follows from 3), and the inequality follows from the fact that $|S_1|=x(G_1)$ and $|S_2|=x(G_2)$. Now, let $S_1^*\in \arg p(G_1)$ and $S_2^*\in\arg p(G_2)$. Then,
\begin{eqnarray*}
p(G_1 \dot\cup G_2) &=& \min_{\substack{S' \in X(G_1 \dot\cup G_2) \\ |S'| = x(G_1 \dot\cup G_2)}} p(G_1 \dot\cup G_2;S')\leq p(G_1\dot\cup G_2;S_1^* \dot\cup S_2^*)\\ 
&=&\max \{p(G_1;S_1^*), p(G_2;S_2^*)\} = \max \{p(G_1), p(G_2) \},
\end{eqnarray*} 
where the inequality follows from 2) and the fact that $x$ is additive, and the second equality follows from 3).
Thus, $p(G_1 \dot\cup G_2) = \max \{p(G_1), p(G_2)\}$, so $p$ is maxing by definition.\newline

Next we will show that \textsc{$(X,p)$-Throttling} is in NP. By 1), for any $S \in X(G)$, $p(G; S)$ can be computed in polynomial time. By 4), \textsc{Minimum $X$ Set} is in NP, so there exists a polynomial time algorithm to verify that $S$ is in $X(G)$. Thus, for any $S \subset V(G)$, $|S|+p(G;S)$ can be computed or found to be undefined in polynomial time. Therefore, \textsc{$(X,p)$-Throttling} is in NP.

We will now show that \textsc{$(X,p)$-Throttling} is NP-Hard, by providing a polynomial reduction from \textsc{Minimum $X$ Set}. Let $\langle G,k\rangle$ be an instance of \textsc{Minimum $X$ Set}. Let $B = b + 1$, where $b$ is the bound on $p(G;S)$ in 1). Let $G_1,\ldots,G_B$ be disjoint copies of $G$, and let $G' = \dot\cup_{i=1}^B G_i$. We will show $\langle G,k\rangle$ is a `yes'-instance of \textsc{Minimum $X$ Set} if and only if $\langle G',Bk + b \rangle$ is a `yes'-instance of \textsc{$(X,p)$-Throttling}. Note that by 1), $\langle G',Bk + b \rangle$ can be constructed in a number of steps that is polynomial in $n$.
Since $x$ is an additive parameter, $x(G')=x(\dot\cup_{i=1}^B G_i)=\sum_{i=1}^B x(G_i)=Bx(G)$. Thus,
\begin{eqnarray*}
\min_{S' \in X(G')} \{|S'| + p(G';S')\} &\leq& \min_{\substack{S' \in X(G') \\ |S'| = x(G')}} \{|S'| + p(G';S')\}\\
&=& \min_{\substack{S' \in X(G') \\ |S'| = x(G')}} \{Bx(G) + p(G';S')\}\\
&=& Bx(G) + p(G') = Bx(G) + p(G),
\end{eqnarray*}
where the last equality follows from the fact that $p$ is maxing, and $p(G') = p(\dot\cup_{i=1}^B G_i) = \max\{p(G_1),\ldots,p(G_B)\} = p(G)$. 

Now consider any $S' \in X(G')$. Clearly $|S'| \geq x(G') = Bx(G)$. Suppose first that $|S'| \geq B(x(G)+1)$; then, 
\[ |S'| + p(G';S') \geq B(x(G)+1) + p(G';S') \geq Bx(G) + B > Bx(G) + p(G).\]

\noindent Now suppose that $|S'| < B(x(G)+1)$. Since $S' \in X(G') = \{\dot\cup_{i=1}^B S_i : S_i \in X(G_i) \}$, $|S' \cap V(G_i)| \geq x(G)$ for all $i\in\{1,\ldots,B\}$. By the pigeonhole principle, $|S' \cap V(G_j)|=|S_j| = x(G)$ for some $j\in\{1,\ldots,B\}$. By 3),
$$ p(G';S') = \max \{p(G_j; S_j), p(G'-G_j; S' \backslash S_j) \} \geq p(G_j; S_j) \geq p(G). $$
Thus in all cases, $|S'| + p(G';S') \geq Bx(G) + p(G)$. Hence, it follows that
\begin{equation}
\label{eq_comp}
\min_{S' \in X(G')}\{|S'| + p(G';S')\} = Bx(G) + p(G).
\end{equation} 
We will now show that $x(G) < k$ if and only if $\min_{S' \in X(G')}\{|S'| + p(G';S')\} < Bk + b$. First, suppose that $x(G) < k$. Then by \eqref{eq_comp}, $\min_{S' \in X(G')}\{|S'| + p(G';S')\} = Bx(G) + p(G) < Bk + b$. Now suppose that $\min_{S' \in X(G')}\{|S'| + p(G';S')\} < Bk + b$. Then, by \eqref{eq_comp}, $Bx(G) + p(G) < Bk + b$. Rearranging, dividing by $B$, and taking the floor yields 
\[x(G)=\lfloor x(G)\rfloor< \left\lfloor k+\frac{b - p(G)}{B} \right\rfloor=k+\left\lfloor\frac{B-1 - p(G)}{B} \right\rfloor = k.\]
Thus, $\langle G,k\rangle$ is a `yes'-instance of \textsc{Minimum $X$ Set} if and only if $\langle G',Bk + b \rangle$ is a `yes'-instance of \textsc{$(X,p)$-Throttling}.
\end{proof}

\noindent We now show that Theorem \ref{thm:comp} can be applied to the specific throttling problems posed for power domination, zero forcing, and positive semidefinite zero forcing.

\begin{corollary}
\label{cor_comp}
\textsc{Power Domination Throttling}, \textsc{Zero Forcing Throttling}, and \textsc{PSD Zero Forcing Throttling} are NP-Complete. 
\end{corollary}

\begin{proof}
Given a graph $G$, let $X(G)$ denote the set of power dominating sets of $G$ and for $S\in X(G)$, let $p(G;S)$ denote the power propagation time of $G$ using $S$. Clearly, for any power dominating set $S$, $\ppt(G;S)$ is bounded above by $|V(G)|$, and can be computed in polynomial time. Thus, assumption 1) of Theorem \ref{thm:comp} is satisfied. For any graphs $G_1$ and $G_2$, it is easy to see that $S$ is a power dominating set of $G_1 \dot\cup G_2$ if and only if $S \cap V(G_1)$ is a power dominating set of $G_1$ and $S \cap V(G_2)$ is a power dominating set of $G_2$. Thus, assumption 2) of Theorem \ref{thm:comp} is satisfied.
Let $G_1$ and $G_2$ be graphs, and let $S_1$ be a power dominating set of $G_1$ and $S_2$ be a power dominating set of $G_2$. Then, the same vertices which are dominated in $G_1$ by $S_1$ and in $G_2$ by $S_2$ can be dominated in $G_1\dot\cup G_2$ by $S_1\dot\cup S_2$, and all forces that occur in timestep $i\geq 2$ in $G_1$ and $G_2$ will occur in $G_1\dot\cup G_2$ at the same timestep. Thus, $\ppt(G_1 \dot\cup G_2; S_1 \dot\cup S_2) = \max\{ \ppt(G_1;S_1), \ppt(G_2;S_2) \}$, so assumption 3) of Theorem \ref{thm:comp} is satisfied. Finally, since \textsc{Power Domination} is NP-Complete (cf. \cite{powerdom3}), assumption 4) of Theorem \ref{thm:comp} is satisfied. Thus, \textsc{Power Domination Throttling} is NP-Complete. By a similar reasoning, it can be shown that the assumptions of Theorem \ref{thm:comp} also hold for zero forcing and positive semidefinite zero forcing; thus, \textsc{Zero Forcing Throttling} and \textsc{PSD Zero Forcing Throttling} are also NP-Complete. 
\end{proof}

Some graph properties are preserved under disjoint unions; we will call a graph property $P$ {\em additive} if for any two graphs $G_1$, $G_2$ with property $P$, $G_1 \dot\cup G_2$ also has property $P$. Let $\left<G, k\right>$ be an instance of \textsc{Minimum $X$ Set} in the special case that $G$ has property $P$. In the proof of Theorem \ref{thm:comp}, a polynomial reduction from $\left<G, k\right>$ to an instance $\left<G', Bk + b\right>$ of \textsc{$(X, p)$-Throttling} is given, where $G'$ is the disjoint union of copies of $G$. If property $P$ is additive, then $G'$ also has property $P$. Thus, special cases of \textsc{$(X, p)$-Throttling} in graphs with property $P$ reduce from instances of \textsc{Minimum $X$ Set} with property $P$, by the proof of Theorem \ref{thm:comp}. It is easy to see that planarity, chordality, and bipartiteness are additive properties. As noted at the beginning of this section, \textsc{Power Domination} is NP-Complete for graphs with these properties. Thus, these NP-Completeness results can be extended to the corresponding throttling problem.

\begin{corollary}
\textsc{Power Domination Throttling} is NP-Complete even for planar graphs, chordal graphs, and bipartite graphs. 
\end{corollary}

\section{Bounds and exact results for $\zpth(G)$}

In this section, we derive several tight bounds and exact results for the power domination throttling number of a graph. We have
also implemented a brute force algorithm for computing the power domination throttling number of arbitrary graphs (cf. \url{https://github.com/rsp7/Power-Domination-Throttling}), and used it to
compute the power domination throttling numbers of all graphs on fewer than 10 vertices. Recall the following well-known bound on the power propagation time.
\begin{lemma}[\cite{ferrero,proptime1}]\label{pth_bound_lemma}
Let $G$ be a graph and $S$ be a power dominating set of $G$. Then 
$\ppt(G;S)\geq \frac{1}{\Delta}\left(\frac{n}{|S|}-1\right)$.
\end{lemma}

\begin{theorem}\label{pth_bound}
Let $G$ be a nonempty graph. Then, $\zpth(G)\ge\big\lceil2\sqrt{\frac{n}{\Delta}}-\frac{1}{\Delta}\big\rceil$, and this bound is tight.
\end{theorem}
\proof
Since $G$ is nonempty, we have $\Delta>0$. Let $\mathcal{P}(G)$ denote the set of all power dominating sets of $G$. By Lemma \ref{pth_bound_lemma}, 
\[\zpth(G)=\min_{S\in \mathcal{P}(G)}\{|S|+\ppt(G;S)\}\ge\min_{S\in \mathcal{P}(G)}\left\{|S|+\frac{1}{\Delta}\left(\frac{n}{|S|}-1\right)\right\}\ge\min_{s>0}\left\{s+\frac{1}{\Delta}\left(\frac{n}{s}-1\right)\right\}.\]
To compute the last minimum, let us minimize $t(s)\vcentcolon=s+\frac{1}{\Delta}(\frac{n}{s}-1),s>0$. Since $t'(s)=1-\frac{n}{\Delta s^2}$, $s=\sqrt{\frac{n}{\Delta}}$ is the only critical point of $t(s)$. Since $t''(s)=\frac{2n}{\Delta s^3}>0$ for $s>0$, we have that $t(\sqrt{\frac{n}{\Delta}})=\sqrt{\frac{n}{\Delta}}+\frac{1}{\Delta}(n/\sqrt{\frac{n}{\Delta}}-1)=2\sqrt{\frac{n}{\Delta}}-\frac{1}{\Delta}$ is the global minimum of $t(s)$. Thus, 
$\zpth(G)=\lceil\zpth(G)\rceil\geq \left\lceil 2\sqrt{\frac{n}{\Delta}}-\frac{1}{\Delta}\right\rceil$.
The bound is tight, e.g., for paths and cycles; see Proposition \ref{pth_path}.
\qed

\begin{theorem}[\cite{carlson2}]\label{thm_zth_path_cycle}
$\zth_{+}(P_n)=\big\lceil\sqrt{2n}-\frac{1}{2}\big\rceil$ for $n\ge1$ and $\zth_{+}(C_n)=\big\lceil\sqrt{2n}-\frac{1}{2}\big\rceil$ for $n\ge4$.
\end{theorem}

\begin{proposition}\label{pth_path}
$\zpth(P_n)=\big\lceil\sqrt{2n}-\frac{1}{2}\big\rceil$ for $n\geq 1$ and $\zpth(C_n)=\big\lceil\sqrt{2n}-\frac{1}{2}\big\rceil$ for $n\geq 3$.
\end{proposition}
\proof
Let $S$ be an arbitrary nonempty subset of $V(P_n)$. If any vertex in $S$ has two neighbors which are not in $S$, then both of these neighbors are in different components of $P_n-S$. Moreover, each vertex in $N[S]$ has at most one neighbor which is not in $N[S]$. Thus, the PSD zero forcing color change rules and the power domination color change rules both dictate that at each timestep, the non-colored neighbors of every colored vertex of $P_n$ will be colored. Hence, since any nonempty subset $S$ of $V(P_n)$ is both a power dominating set and a PSD zero forcing set, $\ppt(P_n;S)=\zpt_+(P_n;S)$. Thus, $\zpth(P_n)=\min\{|S|+\ppt(P_n;S)\colon S\subset V(P_n),|S|\ge1\}=\min\{|S|+\zpt_+(P_n;S)\colon S\subset V(P_n),|S|\ge1\}=\zth_+(P_n)=\big\lceil\sqrt{2n}-\frac{1}{2}\big\rceil$, where the last equality follows from Theorem \ref{thm_zth_path_cycle}.

Clearly $\zpth(C_n)=\big\lceil\sqrt{2n}-\frac{1}{2}\big\rceil$ for $n=3$, so suppose that $n\ge4$. By a similar reasoning as above, and since any set $S\subset V(C_n)$ of size at least 2 is both a power dominating set and a PSD zero forcing set, it follows that $\ppt(P_n;S)=\zpt_+(P_n;S)$. If $\{v\}\subset V(C_n)$ is a power throttling set of $C_n$ and $u$ is a vertex of $C_n$ at maximum distance from $v$, then $\{u,v\}$ is also a power throttling set, since $\ppt(C_n;\{u,v\})\le\ppt(C_n;\{v\})-1$ for $n\ge4$. Thus, $\zpth(C_n)=\min\{|S|+\ppt(C_n;S)\colon S\subset V(C_n),|S|\ge1\}=\min\{|S|+\ppt(C_n;S)\colon S\subset V(C_n),|S|\ge2\}=\min\{|S|+\zpt_+(C_n;S)\colon S\subset V(C_n),|S|\ge2\}=\zth_+(C_n)=\big\lceil\sqrt{2n}-\frac{1}{2}\big\rceil$, where the last equality follows from Theorem \ref{thm_zth_path_cycle}.
\qed

\begin{proposition}\label{pth_disjoint_union}
Let $G_1,G_2$ be graphs and $G=G_1\dot\cup G_2$. Then, 
\begin{eqnarray*}
\zpth(G)&\geq& \max\{\gamma_P(G_1)+\zpth(G_2),\gamma_P(G_2)+\zpth(G_1)\},\\
\zpth(G)&\le&\gamma_P(G_1)+\gamma_P(G_2)+\max\{\ppt(G_1),\ppt(G_2)\},
\end{eqnarray*}
and these bounds are tight.
\end{proposition}
\proof
We first establish the lower bound. Suppose for contradiction that $\zpth(G)<\gamma_P(G_1)+\zpth(G_2)$, and let $S$ be a power throttling set of $G$. Thus, $|S|+\ppt(G;S)<\gamma_P(G_1)+\zpth(G_2)$. Note that $|S\cap V(G_2)|\le|S|-\gamma_P(G_1)$, since $S\cap V(G_1)$ must be a power dominating set of $G_1$. Moreover, $\ppt(G_2;S\cap V(G_2))\le\ppt(G;S)$. Thus,
\begin{eqnarray*}
\zpth(G_2)&\le&|S\cap V(G_2)|+\ppt(G_2;S\cap V(G_2))\\
&\le&|S|-\gamma_P(G_1)+\ppt(G;S)\\
&<&\zpth(G_2),
\end{eqnarray*} a contradiction. Thus, $\zpth(G)\geq \gamma_P(G_1)+\zpth(G_2)$. Similarly, $\zpth(G)\geq \gamma_P(G_2)+\zpth(G_1)$.
We now establish the upper bound. Let $S_1\subset V(G_1)$ and $S_2\subset V(G_2)$ be power dominating sets such that $\ppt(G_1;S_1)=\ppt(G_1)$ and $\ppt(G_2;S_2)=\ppt(G_2)$. Let $S=S_1\cup S_2$. Then $\zpth(G)\le|S|+\ppt(G;S)=|S_1|+|S_2|+\max\{\ppt(G_1;S_1)+\ppt(G_2;S_2)\}=\gamma_P(G_1)+\gamma_P(G_2)+\max\{\ppt(G_1),\ppt(G_2)\}$. Both bounds are tight, e.g., when $G$ is the disjoint union of two stars.\qed
\vspace{9pt}

\begin{theorem} \label{thm:th-ub}
Let $G_1$ and $G_2$ be graphs such that $G_1 \cap G_2\simeq K_k$. Then \[\max\{\zpth(G_1),\zpth(G_2)\}\leq \zpth(G_1 \cup G_2) \leq \gamma_P(G_1) + \gamma_P(G_2) + k + \max\{\ppt(G_1),\ppt(G_2)\},\]
and these bounds are tight.
\end{theorem}
\proof
Let $K=V(G_1 \cap G_2)$. We will first establish the upper bound.  Let $S_1\subset V(G_1)$ and $S_2\subset V(G_2)$ be minimum power dominating sets such that $\ppt(G_1;S_1)=\ppt(G_1)$ and $\ppt(G_2;S_2)=\ppt(G_2)$. Let $S = S_1 \cup S_2 \cup K$. $S$ is a power dominating set of $G_1\cup G_2$, since all vertices which are dominated in $G_1$ by $S_1$ and in $G_2$ by $S_2$ are dominated in $G_1\cup G_2$ by $S_1\cup S_2$, and all forces which occur in $G_1$ and in $G_2$ can also occur in $G_1\cup G_2$ (or are not necessary); this is because $N[K]$ is colored after the domination step, and the non-colored neighbors of any vertex $v \in V(G_1\cup G_2)$ at any forcing step  are a subset of the non-colored neighbors of $v$ at the same timestep in $G_1$ or $G_2$. For the same reason, a force which occurs in timestep $i\geq 2$ in $G_1$ or $G_2$ occurs in a timestep $j\leq i$ in $G_1\cup G_2$ (or is not necessary). Therefore, $\ppt(G_1 \cup G_2; S) \leq \max \{\ppt(G_1),\ppt(G_2)\}$, and $|S|\leq \gamma_P(G_1)+\gamma_P(G_2)+k$. Thus, $\zpth(G_1 \cup G_2) \leq |S|+\ppt(G_1\cup G_2;S)\leq \gamma_P(G_1) + \gamma_P(G_2) + k + \max \{\ppt(G_1),\ppt(G_2)\}$.

We will now establish the lower bound.
Let $S$ be a power throttling set of $G_1 \cup G_2$ and let $w$ be any vertex in $K$. We will consider four cases.

\emph{Case 1:} $S\cap K\neq \emptyset$. In this case, let $S_1=S\cap V(G_1)$ and $S_2=S\cap V(G_2)$.

\emph{Case 2:} $S\cap K = \emptyset$ but $S\cap V(G_1)\neq \emptyset$ and $S\cap V(G_2)\neq \emptyset$. In this case, let $S_1=(S\cap V(G_1))\cup \{w\}$ and $S_2=(S\cap V(G_2))\cup\{w\}$.

\emph{Case 3:} $S\subset V(G_1)\backslash V(G_2)$. In this case, let $S_1=S$ and $S_2=\{w\}$.

\emph{Case 4:} $S\subset V(G_2)\backslash V(G_1)$. In this case, let $S_1=\{w\}$ and $S_2=S$.

\noindent Note that in all cases, $S_1 \subset V(G_1)$, $S_2 \subset V(G_2)$, $|S_1|\leq |S|$, and $|S_2|\leq |S|$. In Cases 1 and 2, $K$ is dominated by $S_1$ in $G_1$ and by $S_2$ in $G_2$. Subsequently, at any forcing step, the non-colored neighbors of any vertex $v$ in $G_1$ or $G_2$ are a subset of the non-colored neighbors of $v$ at the same timestep in $G_1\cup G_2$. Thus, $S_1$ is a power dominating set of $G_1$ and $S_2$ is a power dominating set of $G_2$. Moreover, a force which occurs in timestep $i\geq 2$ in $G_1\cup G_2$ occurs in a timestep $j\leq i$ in $G_1$ or $G_2$. Therefore, $\ppt(G_1; S_1) \leq \ppt(G_1 \cup G_2; S)$, and $\ppt(G_2; S_2) \leq \ppt(G_1 \cup G_2; S)$. 
In Case 3, since no vertex of $K$ is in $S$, no vertex of $K$ colors another vertex of $G_1\cup G_2$ in the domination step. Thus, in $G_1\cup G_2$, no vertex  in $V(G_2)\backslash K$ can force a vertex of $K$, since this would mean a vertex in $K$ forced some vertex in $V(G_2)\backslash K$ in a previous timestep, which would require all vertices of $K$ to already be colored. Moreover, in $G_1\cup G_2$, all vertices in $V(G_2)\backslash K$ can be forced after the vertices in $K$ get colored. Thus, $S_1$ is a power dominating set of $G_1$ and $S_2$ is a power dominating set of $G_2$. Furthermore, since $S_1$ and $S_2$ can color $G_1$ and $G_2$ using a subset of the forces that are used by $S$ to color $G_1\cup G_2$, it follows that $\ppt(G_1; S_1) \leq \ppt(G_1 \cup G_2; S)$ and $\ppt(G_2; S_2) \leq \ppt(G_1 \cup G_2; S)$. Case 4 is symmetric to Case 3. 
Thus, in all cases, $\zpth(G_1)\leq |S_1| + \ppt(G_1; S_1) \leq |S| + \ppt(G_1 \cup G_2; S)=\zpth(G_1\cup G_2)$ and $\zpth(G_2)\leq |S_2| + \ppt(G_2; S_2) \leq |S| + \ppt(G_1 \cup G_2; S)=\zpth(G_1\cup G_2)$, so $\max\{\zpth(G_1),\zpth(G_2)\}\leq \zpth(G_1 \cup G_2)$.

To see that the upper bound is tight, let $K$ be a complete graph with vertex set $\{v_1,\ldots,v_k\}$, let $G_1$ be the graph obtained by appending two leaves, $u_i$ and $w_i$, to each vertex $v_i$ of $K$, $1\leq i\leq k$, and then appending three paths of length 1 to each $w_i$, $1\leq i\leq k$. Let $G_2$ be a copy of $G_1$ labeled so that $G_1\cap G_2=K$ and the vertex in $G_2$ corresponding to $w_i$ in $G_1$ is $w_i'$, $1\leq i\leq k$; see Figure \ref{fig:th-upperbounds} for an illustration. Let $S = \{w_1,\ldots,w_k\}$. Since every minimum power dominating set of $G_1$ must contain $S$, and $S$ is itself a power dominating set of $G_1$, $\gamma_P(G_1) = \gamma_P(G_2) = |S| = k$. Furthermore, $\max\{\ppt(G_1),\ppt(G_2)\}=2$, so $\gamma_P(G_1) + \gamma_P(G_2) + k + \max\{\ppt(G_1),\ppt(G_2)\}  = 3k + 2$. In $G_1 \cup G_2$, for $1\leq i\leq k$, $v_i$ has two leaves appended to it; thus, either $v_i$ or one of these two leaves must be contained in any power dominating set of $G_1 \cup G_2$. Likewise, since each vertex $w_i$ has three paths appended to it, either $w_i$ or at least one vertex in those paths must be contained in any power dominating set. Similarly, either $w_i'$ or at least one vertex in the paths appended to $w_i'$ must be contained in any power dominating set.
Thus, $\gamma_P(G_1 \cup G_2) \geq 3k$. If $\zpth(G_1 \cup G_2) \leq 3k+1$, then there must exist a power dominating set $S'$ such that $\ppt(G_1 \cup G_2;S') = 1$, and $|S'|=3k$. However, if $\ppt(G_1 \cup G_2;S') = 1$, then $S'$ must be a dominating set, and it is easy to see that $G_1\cup G_2$ does not have a dominating set of size $3k$. Therefore $\zpth(G_1 \cup G_2) = 3k+2=\gamma_P(G_1) + \gamma_P(G_2) + k + \max\{\ppt(G_1),\ppt(G_2)\}$.

\begin{figure}[H]
\centering
\includegraphics[scale=.4]{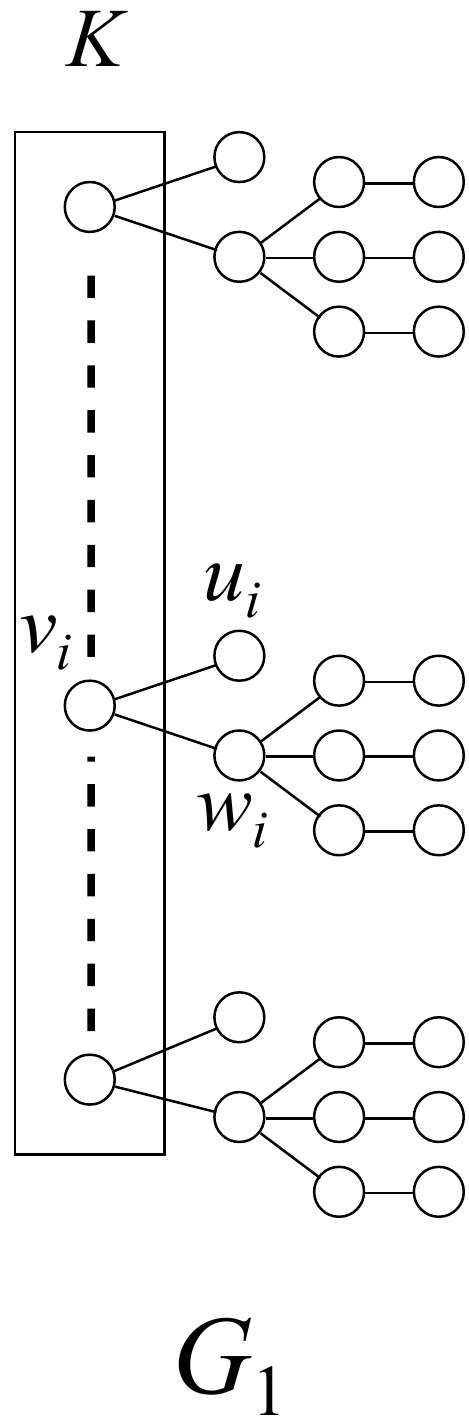} \hspace{2cm} \includegraphics[scale=.4]{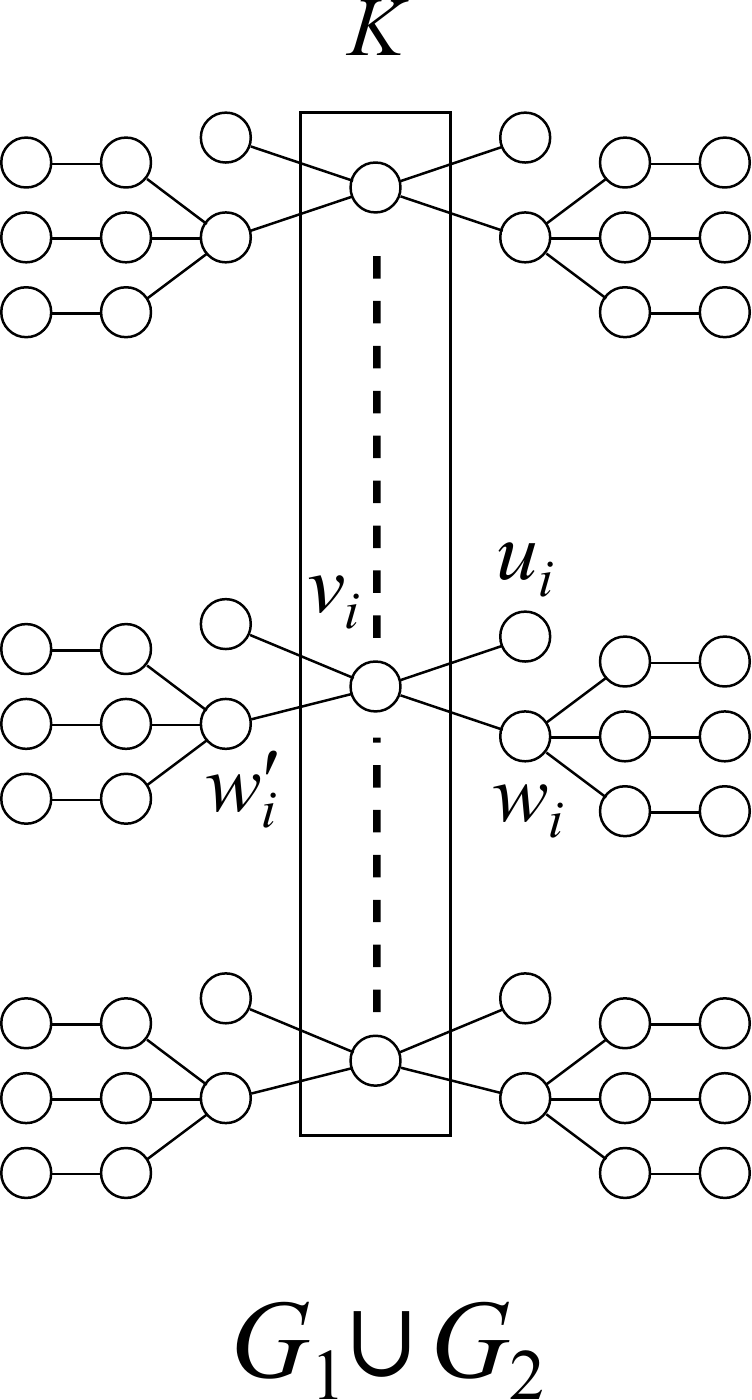}
\caption{Graphs $G_1$ and $G_1 \cup G_2$ for which the upper bound in Theorem \ref{thm:th-ub} holds with equality.}
\label{fig:th-upperbounds}
\end{figure}

To see that the lower bound is tight, let $K$ be a complete graph on $k$ vertices, let $G_1$ be the graph obtained by appending three leaves to each vertex of $K$, and let $G_2$ be a copy of $G_1$ labeled so that $G_1\cap G_2=K$. Then, $V(K)$ is a power throttling set of $G_1$, $G_2$ and $G_1 \cup G_2$, since $V(K)$ is a minimum power dominating set in all three graphs, and the power propagation time in all three graphs using $V(K)$ is 1. Thus, $\zpth(G_1 \cup G_2)=k+1=\max\{\zpth(G_1),\zpth(G_2)\}$.
\qed
\vspace{9pt}

\noindent We conclude this section by deriving tight bounds on the power domination throttling numbers of trees; some ideas in the following results are adapted from \cite{carlson2}.

\begin{lemma}\label{lemma_no_leaves}
Let $G$ be a connected graph on at least $3$ vertices. Then there exists a power throttling set of $G$ that contains no leaves.
\end{lemma}
\proof
Let $S'$ be a power throttling set of $G$, and suppose that $v\in S'$ is a leaf with neighbor $u$ (which cannot be a leaf since $G$ is connected and $n(G)\geq 3$). If $u\in S'$, then $S\vcentcolon=S'\setminus\{v\}$ is also a power throttling set of $G$, since $|S|=|S'|-1$ and $\ppt(G;S)\le\ppt(G;S')+1$. Otherwise, if $u\notin S'$, then let $S=(S'\setminus\{v\})\cup\{u\}$. Note that $N[S']\subset N[S]$, and so $\zpt(G;N[S])\le\zpt(G;N[S'])$. Since $\ppt(G;S),\ppt(G;S')\ge1$, this implies that $\ppt(G;S)\le\ppt(G;S')$. Since $|S|=|S'|$, $S$ must also achieve throttling. This process of replacing leaves with non-leaf vertices in power throttling sets of $G$ can be repeated until a power throttling set is obtained which has no leaves.

\qed

\begin{proposition}\label{prop_subtree_monotone}
	If $T$ is a tree with subtree $T'$, then $\zpth(T')\le\zpth(T)$. That is, power domination throttling is subtree monotone for trees.
\end{proposition}
\proof
Clearly the claim is true for trees with at most $2$ vertices, so suppose that $T$ is a tree with at least $3$ vertices. 
By Lemma \ref{lemma_no_leaves}, $T$ has a power throttling set $S$ which does not contain leaves. Let $v$ be a leaf of $T$; then, $S\subset V(T-v)$, so $\ppt(T-v;S)\le\ppt(T;S)$. Thus, $\zpth(T-v)\leq |S|+\ppt(T-v;S)\leq |S|+\ppt(T;S)=\zpth(T)$. Since any subtree $T'$ of $T$ can be attained by repeated removal of leaves, and since each removal of a leaf does not increase the power domination throttling number, it follows that $\zpth(T')\le\zpth(T)$.
\qed

\begin{theorem}\label{cor_tree_diam_bound}
	Let $T$ be a tree on at least $3$ vertices. Then, \[\big\lceil\sqrt{2(diam(T)+1)}-1/2\big\rceil\le\zpth(G)\le diam(T)-1+\gamma_P(T),\] and these bounds are tight.
\end{theorem}
\proof
Since $T$ has diameter $d:=diam(T)$ and at least 3 vertices, $T$ contains a path of length $d\geq 2$. Thus $P_{d+1}$ is a subtree of $T$, and $\Delta(P_{d+1})=2$. Then, the lower bound follows from Theorem~\ref{pth_bound} and Proposition \ref{prop_subtree_monotone}. In Theorem 2.5 of \cite{ferrero}, it is shown that for every tree with at least 3 vertices, $\ppt(T)\le d-1$. Let $S^*$ be a power throttling set of $T$ and $S$ be a minimum power dominating set of $T$ such that $\ppt(T;S)=\ppt(T)$. Then, $\zpth(T)=|S^*|+\ppt(T;S^*)\leq |S|+\ppt(T;S)=\gamma_P(T)+\ppt(T)\leq \gamma_P(T)+d-1$. Both bounds are tight, e.g., for stars, since $\big\lceil\sqrt{2(2+1)}-1/2\big\rceil= 2-1+1$.
\qed
\vspace{9pt}

\section{Extremal power domination throttling numbers}

In this section, we give a characterization of graphs whose power domination throttling number is at least $n-1$ or at most $t$, for any constant $t$.
We begin by showing that graphs with $\zpth(G)\leq t$ are minors of the graph in the following definition.

\begin{definition}
\label{def_gsab}
Let $a\geq 0$, $b\geq 0$, and $s\geq 1$ be integers and let $\pd(s,a,b)$ be the graph obtained from $K_s \dot{\cup} (K_a \square P_b)$ by adding every possible edge between the disjoint copy of $K_s$ and a copy of $K_a$ in $K_a \square P_b$ whose vertices have minimum degree. If either $a=0$ or $b=0$, then $\pd(s,a,b)\simeq K_s$.
A \emph{path edge} of $\pd(s,a,b)$ is an edge that belongs to one of the copies of $P_b$; a \emph{complete edge} is an edge that belongs to one of the copies of $K_a$, or to $K_s$; a \emph{cross edge} is an edge between $K_s$ and $K_a \square P_b$. The vertices in $K_s$ and $K_a$ that are incident to cross edges are called \emph{$s$-vertices} and \emph{$a$-vertices}, respectively. See Figure \ref{PDgraph} for an illustration.
\end{definition}

\begin{figure}[H] \begin{center}
\scalebox{.75}{\includegraphics{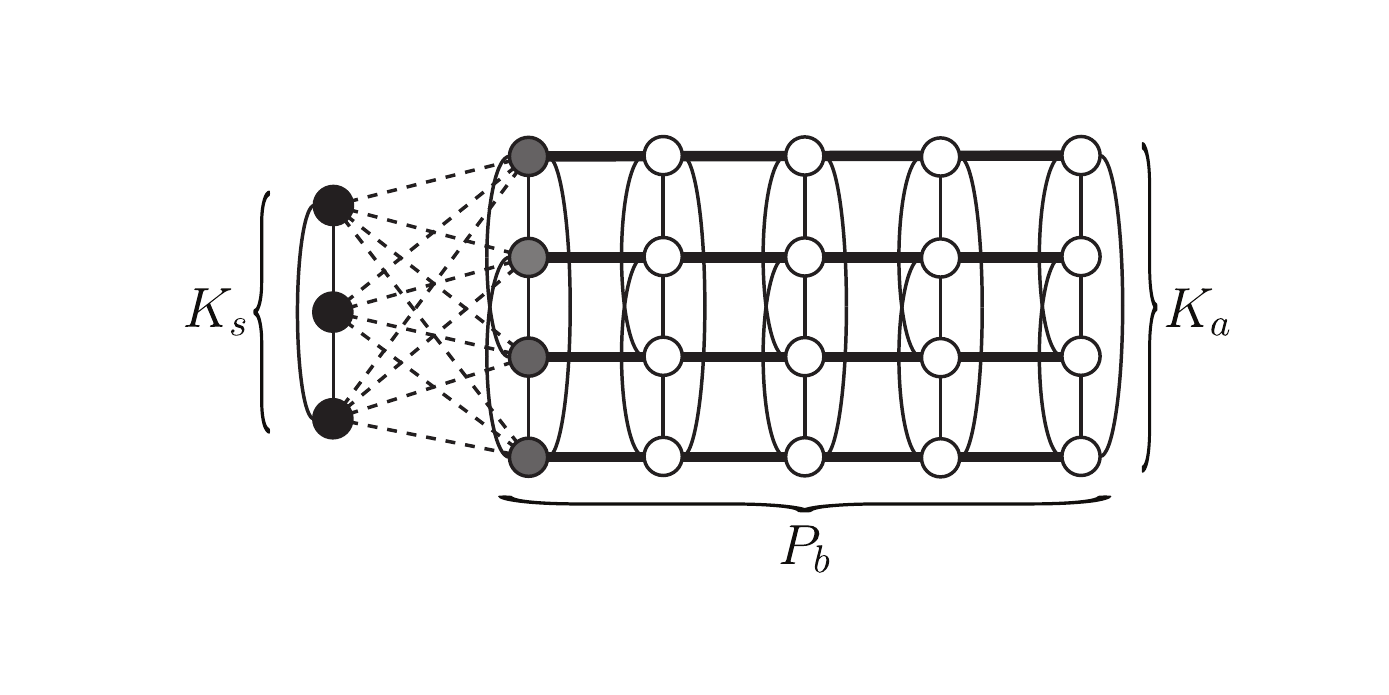}}\\
\caption{The graph $\pd(s,a,b)$ with $s=3$, $a = 4$, and $b = 5$. The dashed edges are the cross edges, the solid edges are the complete edges, the thick edges are the path edges, the black vertices are $s$-vertices and the grey vertices are $a$-vertices.}\label{PDgraph}
\end{center}
\end{figure}

\begin{theorem}\label{pdChar}
Let $G$ be a graph and $t$ be a positive integer. Then, $\zpth(G) \leq t$ if and only if there exist integers $a\geq 0$, $b\geq 0$, and $s\geq 1$ such that $s + b = t$, and 
$G$ can be obtained from $\pd(s, a, b)$ by
\begin{enumerate}
\item contracting path edges, 
\item deleting complete edges, and/or
\item deleting cross edges so that the remaining cross edges saturate the $a$-vertices. 
\end{enumerate}
Moreover, for a fixed $t$, these conditions can be verified in polynomial time.
\end{theorem}
\proof
Suppose first that $\zpth(G) \leq t$. Let $S$ be a power throttling set of $G$, and fix some chronological list of forces by which $N[S]$ colors $G$. 
Let $s = |S|$, let $b' = \ppt(G; S) = \zpth(G) -s$, and let $b = t-s$; note that $b' \leq b$. Let $A = N[S] \setminus S=\{v_{1,1}, v_{2,1},\ldots,v_{a,1}\}$, where $a=|A|$. Clearly, $a \leq s\Delta(G)$. We will show that $G$ can be obtained from $\pd(s,a,b)$ by contracting path edges, deleting complete edges, and/or deleting cross edges so that the remaining cross edges saturate the $a$-vertices. First, note that $\pd(s,a,b')$ can be obtained from $\pd(s,a,b)$ by contracting path edges. Thus, it suffices to show that $G$ can be obtained from $\pd(s,a,b')$ by the above operations.


Label the $s$-vertices of $\pd(s,a,b')$ with the elements of $S$, and label the $a$-vertices of $\pd(s,a,b')$ with the elements of $\{v_{1,1}^1, v_{2,1}^1,\ldots,v_{a,1}^1\}$. For each $s$-vertex $u$ and $a$-vertex $v_{i,1}^1$, delete the edge $uv_{i,1}^1$ unless $uv_{i,1} \in E(G)$. Note that all edges deleted this way are cross edges, and that after these deletions, the remaining cross edges must saturate the $a$-vertices, since by definition the vertices in $S$ dominate the vertices in $A$. Also, for each pair of $s$-vertices $u_1$, $u_2$, delete the edge $u_1u_2$ unless $u_1u_2\in E(G)$; note that all edges deleted this way are complete edges.

For $1\leq i\leq a$, let $v_{i,1},\ldots,v_{i,{p_i}}$ be a maximal sequence of vertices of $G$ such that $v_{i,j}$ forces $v_{i,j+1}$ for $1 \leq j < p_i$ (after the domination step using $S$). Note that since $A=N[S]\backslash S$, $A$ is a zero forcing set of $G-S$, and hence each vertex of $G-S$ belongs to exactly one such sequence. 
For $1 \leq i \leq a$ and $1 \leq j \leq p_i$, if $v_{i,j}$ performs a force, let $\tau_{i,j}$ be the timestep at which $v_{i,j}$ performs a force minus the timestep at which $v_{i,j}$ gets colored;  if $v_{i,j}$ does not perform a force, let $\tau_{i,j}$ be $b'+1$ minus the timestep at which $v_{i,j}$ gets colored. Note that for each $i\in \{1,\ldots,a\}$, $\sum_{j=1}^{p_i}\tau_{i,j}=b'$. Thus, if $P_1,\ldots,P_a$ are the paths used in the construction of $\pd(s,a,b')$, we can label the vertices of $P_i$, $1\leq i\leq a$, in order starting from the endpoint which is an $a$-vertex toward the other endpoint, as 
\[v_{i,1}^1,\ldots,v_{i,1}^{\tau_{i,1}},v_{i,2}^1,\ldots,v_{i,2}^{\tau_{i,2}},v_{i,3}^1,\ldots,v_{i,3}^{\tau_{i,3}},\ldots,v_{i,p_i}^1,\ldots,v_{i,p_i}^{\tau_{i,p_i}}.\] 

Let $K_1,\ldots,K_{b'}$ be the cliques of size $a$ used in the construction of $\pd(s,a,b')$, where $V(K_1)=\{v_{1,1}^1,\ldots,v_{a,1}^1\}$, and the vertices of $K_{\ell}$ are adjacent to the vertices of $K_{\ell+1}$ for $1\leq \ell<b'$. Thus, each such clique corresponds to a timestep in the forcing process of $G-S$ using $A$. Let $e=\{v_{i_1,j_1},v_{i_2,j_2}\}$ be an arbitrary edge of $G-S$ with $i_1 \neq i_2$. There is an earliest timestep $\ell^*$ at which both $v_{i_1,j_1}$ and $v_{i_2,j_2}$ are colored. Therefore, the clique $K_{\ell^*}$ contains $v_{i_1,j_1}^{\alpha}$ and $v_{i_2,j_2}^{\beta}$, for some $\alpha\in \{1,\ldots,\tau_{i_1,j_1}\}$ and $\beta\in \{1,\ldots,\tau_{i_2,j_2}\}$. Denote the edge $\{v_{i_1,j_1}^{\alpha}$,$v_{i_2,j_2}^{\beta}\}$ by $\phi(e)$, and note that $\phi(e)$ is uniquely determined for  $e$.

Delete all edges in $K_1,\ldots,K_{b'}$ from $\pd(s,a,b')$ except the edges $\{\phi(e):e=\{v_{i_1,j_1},v_{i_2,j_2}\}\in E(G-S),  \text{ with }i_1 \neq i_2\}$. Next, for $1 \leq i \leq a$ and $1 \leq j \leq p_i$, contract the edges $\{v_{i,j}^1,v_{i,j}^2\},\{v_{i,j}^2,v_{i,j}^3\},\ldots,\{v_{i,j}^{\tau_{i,j}-1},v_{i,j}^{\tau_{i,j}}\}$ in $\pd(s,a,b')$ and let $\psi(v_{i,j})$ be the vertex corresponding to $\{v_{i,j}^1,\ldots,v_{i,j}^{\tau_{i,j}}\}$ obtained from the contraction of these edges. See Figure \ref{PDextExample} for an illustration. Note that these operations delete complete edges and contract path edges. Moreover, note that there is a bijection between edges of $G-S$ of the form $e=\{v_{i_1,j_1},v_{i_2,j_2}\}$ with $i_1 \neq i_2$ and the edges $\phi(e)$ of $\pd(s,a,b')$, as well as between edges of the form $\{v_{i,j},v_{i,{j+1}}\}$ of $G-S$ and the edges $\{\psi(v_{i,j}),\psi(v_{i,{j+1}})\}$ of $\pd(s,a,b')$. Thus, the obtained graph is isomorphic to $G$, so $G$ can be obtained from $\pd(s,a,b')$ by contracting path edges, deleting complete edges, and/or deleting cross edges so that the remaining cross edges saturate the $a$-vertices.

\begin{figure}[H] 
\begin{center}
\includegraphics[scale=.5]{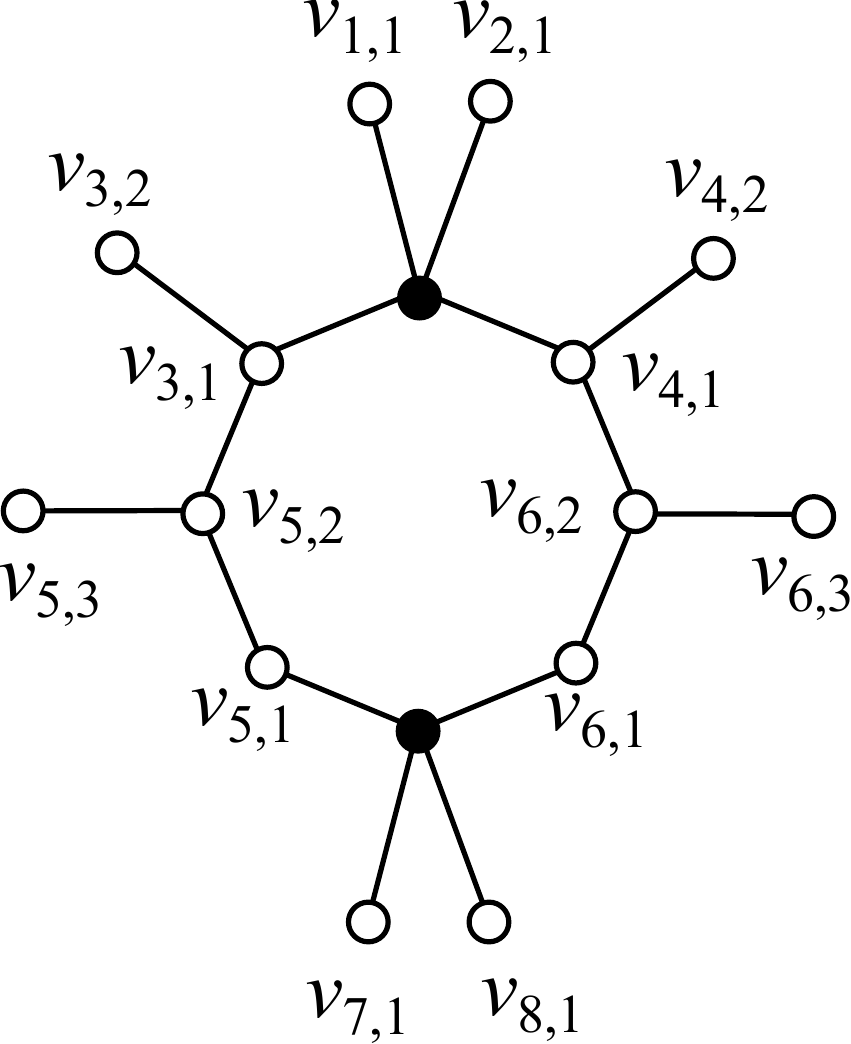} \hspace{50pt}
\includegraphics[scale=.5]{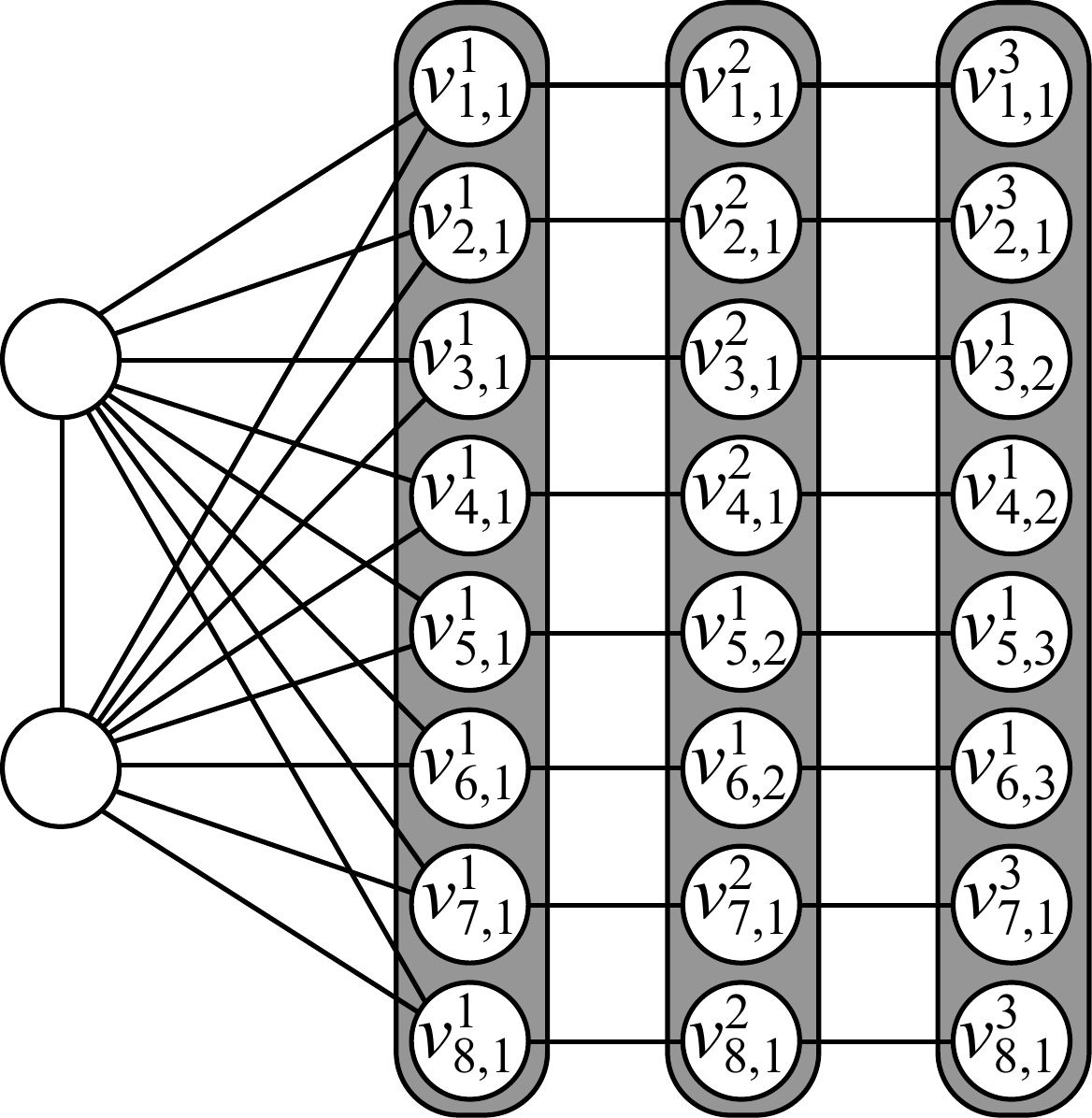} \hspace{50pt}\\
\vspace{20pt}
\includegraphics[scale=.5]{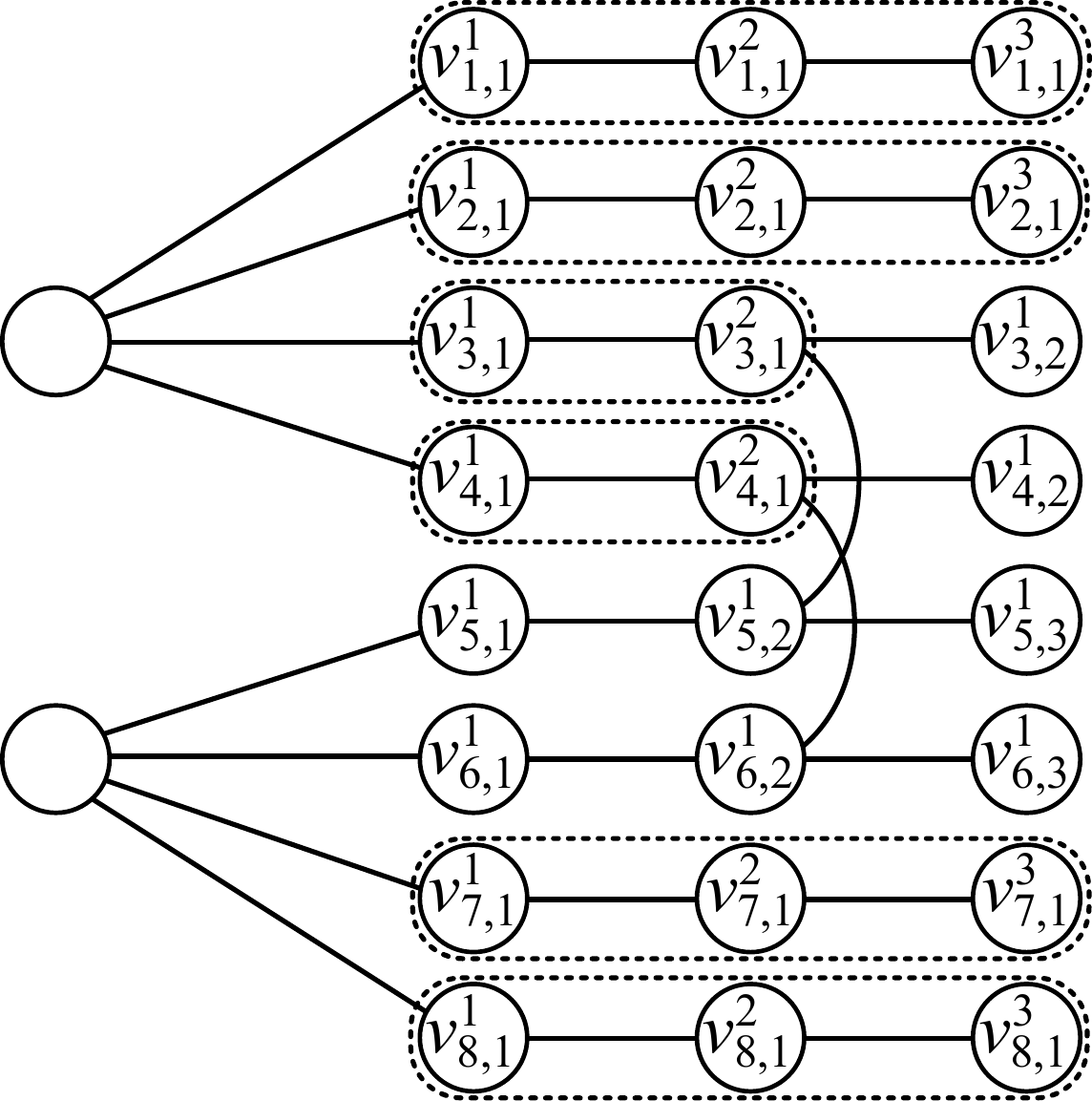} \hspace{50pt}
\includegraphics[scale=.5]{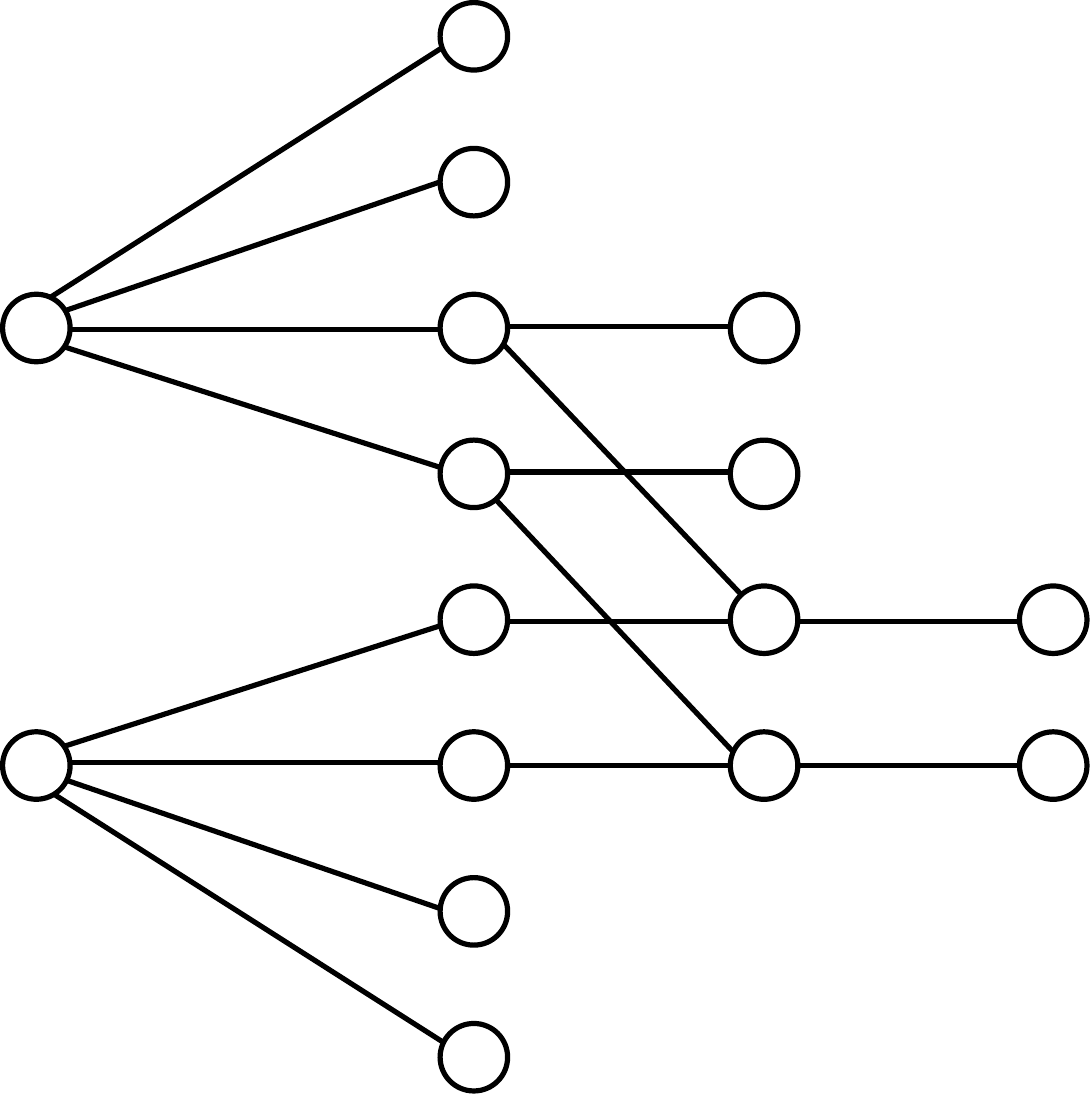}
\caption{\emph{Top left:} A graph $G$; the shaded vertices are a power throttling set of $G$. \emph{Top right:} The graph $\pd(2,8,3)$ is constructed and its vertices are labeled; shaded ovals represent complete edges.  \emph{Bottom left:} The necessary cross edges and complete edges are deleted, and the path edges to be contracted are shown in dashed ovals. \emph{Bottom right:} After the path edges are contracted, the original graph $G$ is obtained.}
\label{PDextExample}
\end{center}
\end{figure}

Conversely, suppose there exist integers $a\geq 0$, $b\geq 0$, and $s\geq 1$ such that $s + b = t$, and $G$ can be obtained from $\pd(s,a,b)$ by contracting path edges, deleting complete edges, and deleting cross edges so that the remaining cross edges saturate the $a$-vertices. Let $S$ be the set of $s$-vertices in $\pd(s,a,b)$ and $A$ be the set of $a$-vertices. Clearly $S$ is a power dominating set of $\pd(s,a,b)$, and $\ppt(\pd(s,a,b);S)=b$.

In the power domination process of $\pd(s,a,b)$ using $S$, complete edges are not used in the domination step and are not used in any forcing step, since any vertex which is adjacent to a non-colored vertex via a complete edge is also adjacent to a non-colored vertex via a path edge. Therefore, $S$ remains a power dominating set after any number of complete edges are deleted from $\pd(s,a,b)$; moreover, deleting complete edges from $\pd(s,a,b)$ cannot increase the power propagation time using $S$, since all the forces can occur in the same order as in the original graph, via the path edges.

It is also easy to see that if any path edges of $\pd(s,a,b)$ are contracted, $S$ remains a power dominating set of the resulting graph, since all the forces can occur in the same relative order along the new paths. 
Moreover, note that $\pd(s,a,b)-S\simeq K_a \square P_b$, and that $A$ is a zero forcing set of $K_a \square P_b$. Thus, the power domination process of $\pd(s,a,b)$ using $S$ after the domination step is identical to the zero forcing process of  $K_a \square P_b$ using $A$. It follows from Lemma 3.15 of \cite{carlson} that contracting path edges of $K_a \square P_b$ does not increase the zero forcing propagation time using $A$. Thus, contracting path edges of $\pd(s,a,b)$ does not increase the power propagation time using $S$.

Finally, deleting cross edges so that the remaining cross edges saturate the $a$-vertices ensures that every $a$-vertex will still be dominated by an $s$-vertex in the first timestep. Thus, since $S$ remains a power dominating set of $G$, and since $G$ is obtained from $\pd(s,a,b)$ by operations that do not increase the power propagation time using $S$, it follows that $\zpth(G) \leq |S|+\ppt(G;S)\leq |S|+\ppt(\pd(s,a,b);S)=s+b=t$.

To see that it can be verified in polynomial time whether a graph $G=(V,E)$ satisfies the conditions of the theorem, note that for a fixed constant $t$, there are $O(n^t)$ subsets of $V$ of size at most $t$. Given a set $S\subset V$, it can be verified in $O(n^2)$ time whether $S$ is a power dominating set of $G$, and if so, $\ppt(G;S)$ can be computed in $O(n^2)$ time. Thus, it can be verified in $O(n^{t+2})$ time whether there exists a power dominating set $S$ with $|S|\leq t-\ppt(G;S)$, and hence whether $\zpth(G)\leq t$.
\qed
\vspace{9pt}

\noindent We can use Theorem \ref{pdChar} to quickly characterize graphs with low power domination throttling numbers. 
\begin{corollary}
Let $G$ be a graph. Then $\zpth(G) = 1$ if and only if $G \simeq K_1$.
\end{corollary}

\begin{corollary}
Let $G$ be a graph. Then $\zpth(G) = 2$ if and only if $G \simeq \overline{K}_2$ or $G$ has a dominating vertex and $G\not\simeq K_1$.
\end{corollary}

\noindent We conclude this section by characterizing graphs whose power domination throttling numbers are large.

\begin{proposition}\label{pth=n}
Let $G$ be a graph. Then $\zpth(G)=n$ if and only if $G\simeq\overline{K}_n$
or $G\simeq K_2\dot\cup\overline{K}_{n-2}$.
\end{proposition}
\proof
If $G\simeq\overline{K}_n$ or $G\simeq K_2\dot\cup\overline{K}_{n-2}$, it is easy to see that $\zpth(G)=n$.
Let $G$ be a graph with $\zpth(G)=n$. If
$|E(G)|=0$, then $G\simeq\overline{K}_n$. If $|E(G)|=1$, then $G\simeq K_2\dot\cup\overline{K}_{n-2}$. If $|E(G)|\geq 2$, then let $u$ and $v$ be distinct endpoints of distinct edges of $G$. Let $S=V\setminus\{u,v\}$, so that
$|S|=n-2$ and  $\ppt(G;S)=1$. This implies that $\zpth(G)\le n-1$, a
contradiction.\qed

\begin{theorem}\label{pth=n-1}
Let $G$ be a graph. Then $\zpth(G)=n-1$ if and only if $G\simeq P_3\dot\cup\overline{K}_{n-3}$ or $G\simeq C_3\dot\cup\overline{K}_{n-3}$ or $G\simeq P_4\dot\cup\overline{K}_{n-4}$ or $G\simeq C_4\dot\cup\overline{K}_{n-4}$ or $G\simeq K_2\dot\cup K_2\dot\cup\overline{K}_{n-4}$.
\end{theorem}
\proof
If $G$ is any of the graphs in the statement of the theorem, then it is easy to see that $\zpth(G)=n-1$. Let $G$ be a graph with $\zpth(G)=n-1$ and suppose $G$ has connected components $G_1,\ldots,G_k$. By Proposition \ref{pth_disjoint_union}, $n(G)-1=\zpth(G)\leq \zpth(G_1)+\ldots+\zpth(G_k)$, so $\zpth(G_i)\geq n(G_i)-1$ for $1\leq i\leq k$.

Let $G_i$ be an arbitrary component of $G$. We will show that $\zpth(G_i)=n(G_i)-1$ if and only if $G_i\in \{P_3,C_3,P_4,C_4\}$.
If $G_i\in \{P_3,C_3,P_4,C_4\}$, then it is easy to see that $\zpth(G_i)=n(G_i)-1$. Now suppose $\zpth(G_i)=n(G_i)-1$. Since $G_i$ is connected and $G_i\not\simeq K_1$, $\Delta(G_i)\ge 1$. If $\Delta(G_i)=1$, then connectedness implies that $G_i\simeq K_2$, but then $\zpth(G_i)=2=n(G_i)$, a contradiction. If $\Delta(G_i)=2$, then connectedness implies that $n(G_i)\ge3$ and $G_i\simeq P_{n(G_i)}$ or $G_i\simeq C_{n(G_i)}$. However, if $n(G_i)\ge 5$, and if we label the vertices of $G_i$ $v_1,\ldots,v_5,\ldots, v_{n(G_i)}$ in order along the path or cycle, then taking $S=V(G_i)\setminus\{v_1,v_3,v_4\}$ yields $\zpth(G_i)\leq |S|+\ppt(G_i;S)=n(G_i)-3+1$, a contradiction. Finally, if $\Delta(G_i)\ge3$ and $v$ is a vertex with $d(v)=\Delta(G_i)$, then taking $S=V(G_i)\setminus N(v)$ yields $\zpth(G_i)\leq|S|+\ppt(G_i;S)\le n(G_i)-2$, a contradiction. Moreover, by Proposition \ref{pth=n}, $\zpth(G_i)=n(G_i)$ if and only if $G_i\in \{K_1, K_2\}$. Thus, each component of $G$ is one of the following: $K_1, K_2, P_3, C_3, P_4, C_4$. 

If one of the components of $G$, say $G_1$, is $P_3$, $C_3$, $P_4$, or $C_4$, then all other components of $G$ must be $K_1$. To see why, let $v$ be a degree 2 vertex in $G_1$, and let $w$ be a non-isolate vertex in another component; then, taking $S=V(G)\setminus(N(v)\cup\{w\})$ yields $\zpth(G)\leq |S|+\ppt(G;S)=n(G)-3+1$, a contradiction. If one of the components of $G$, say $G_1$, is $K_2$, then exactly one other component must be $K_2$, and all other components must be $K_1$. To see why, note that by the argument above, no other component can be $P_3$, $C_3$, $P_4$, or $C_4$, and by Proposition \ref{pth=n}, there must be a component different from $K_1$. Thus, this component must also be a $K_2$ component. If there are at least three $K_2$ components, then let $v_1,v_2,v_3$ be  degree 1 vertices, each belonging to a distinct $K_2$ component; taking $S=V(G)\setminus\{v_1,v_2,v_3\}$ yields $\zpth(G)\leq |S|+\ppt(G;S)=n(G)-3+1$, a contradiction. Thus, there are exactly two $K_2$ components. 
\qed

\section{Conclusion}
In this paper, we presented complexity results, tight bounds, and extremal characterizations for the power domination throttling number. Our complexity results apply not only to power domination throttling, but also to a general class of minimization problems defined as the sum of two graph parameters. One direction for future work is to determine the largest value of $\zpth(G)$ for a connected graph $G$. For example, $\zpth(G)\geq \gamma_P(G)$, and there are graphs for which $\gamma_P(G)=\frac{n}{3}$. Is there an infinite family of connected graphs for which $\zpth(G)=\frac{n}{2}$? It would also be interesting to find operations which affect the power domination throttling number monotonely, or conditions which guarantee that the power domination throttling number of a graph is no less than or no greater than the power domination throttling number of an induced subgraph. We partially answered this question by showing that power domination throttling is subtree monotone for trees. Finding an exact polynomial time algorithm for the power domination throttling number of trees would also be of interest.

\section*{Acknowledgements}
This work is supported by grants CMMI-1634550 and DMS-1720225.

\bibliographystyle{abbrv}

\end{document}